\documentclass[13pt,oneside,a4paper,openany]{article}

\usepackage{amssymb, tikz}
\usepackage{mathrsfs}
\usepackage{dsfont}
\usepackage{amsfonts,enumitem, enumerate, xspace, amsthm}
\usepackage{changebar}
\usepackage{hyperref}
\usepackage{pdfsync}
\usepackage{color}
\usepackage{bm}
\usepackage{amsmath}

\numberwithin{equation}{section}

\newtheorem{thm}{Theorem}[section]

\newtheorem{prop}[thm]{Proposition}

\newtheorem{rem}{Remark}[section]

\newcommand\s{\epsilon}

\def\dsum{\displaystyle\sum}
\def\dlim{\displaystyle\lim}
\def\dint{\displaystyle\int}

\allowdisplaybreaks

\begin{document}

\title{Sticky Brownian motions and a probabilistic solution to a two-point boundary value problem}

\author{Thu Dang Thien Nguyen}


\maketitle

\begin{abstract}
{In this paper, we study a two-point boundary value problem consisting of the heat equation on the open interval $(0,1)$ with boundary conditions which relate first and second spatial derivatives at the boundary points. Moreover, the unique solution to this problem can be represented probabilistically in terms of a sticky Brownian motion. This probabilistic representation is attained from the stochastic differential equation for a sticky Brownian motion on the bounded interval $[0,1]$.
}
\end{abstract}

\maketitle

\section{Introduction}
Let $u_0: [0,1]\to [0,1]$. We try to look for a bounded solution $u\in C^{2,1}([0,1]\times (0,\infty))$ solving the following problem
\begin{align}
 u_t =  \frac 1{2} u_{rr},&\quad\quad \dlim_{t\downarrow 0}u(r,t) =u_0(r), \text{ for } r\in (0,1),\label{ep1}\\
\dfrac{d}{dt}u(0,t)=\dfrac{1}{2}u_r(0,t), &\quad\quad \dfrac{d}{dt}u(1,t)=-\dfrac{1}{2}u_r(1,t)\label{ep2}.
\end{align}

Since it is shown in Chapter 6 and Chapter 7 of \cite{Kt} that there exists a Markov process $\mathbf{B}$ having a generator $A=\frac{1}{2}d^2/dx$ on $(0,1)$ with extension by continuity to the points $0$ and $1$ and restriction to the domain
$$ \mathcal{D}(A)=\{f\in C^2([0,1]), f''(x)+(-1)^{1-x}f'(x)=0 \text{ for } x\in \{0,1\}\}, $$
then if $u_0\in C^2([0,1])$, one can verify that the unique bounded solution to the above problem is given probabilistically, for $(r,t)\in [0,1]\times (0.\infty)$, by
$$ u(r,t)=E_r\big[u_0(\mathbf{B}(t))\big], $$
where $E_r$ stands for the expectation with respect to the process $\mathbf{B}$ starting from $r$.
 
However, when the continuity at the boundary points $0$ and $1$ of the initial datum $u_0$ is relaxed, let us consider the heat equation along with the initial datum \eqref{ep1} and the Dirichlet boundary conditions
\begin{equation}\label{Dir}
u(0,t)=v_{0,-}, \quad\quad u(1,t)=v_{0,+} \text{ for } t>0,
\end{equation}
where $v_{0,\pm}\in [0,1]$. This boundary value problem has a unique solution which can be represented by means of a Brownian motion $B$ absorbed at $0$ and $1$ as follows
\begin{equation}\label{abs}
u(r,t)=E_r\big[u_0(B(t))\mathbf{1}_{\tau>t}\big]+v_{0,-}P_r\big(\tau=\tau_0\leq t\big)+v_{0,+}P_r\big(\tau=\tau_1\leq t\big),
\end{equation}
where $\tau_a$ is the first time when the Brownian motion $B$ hits $a$ and $\tau=\tau_0\wedge\tau_1$. Since $B$ is absorbed whenever it reaches $0$ and $1$, we are just interested in the boundary conditions of the form \eqref{Dir}.

In \cite{PS}, Pang and Stroock investigate the existence of the solution to the heat equation with the initial datum \eqref{ep1} and the boundary conditions \eqref{ep2} under the assumption on the discontinuity of the initial datum $u_0$ at the boundaries. For the boundary conditions \eqref{ep2}, we are concerned about the mass flux at each boundary. The solution in $C^{2,1}([0,1]\times (0,\infty))$ established through a Brownian motion sticky at $0$ and $1$ is unique if it satisfies further that
\begin{equation}\label{ep3}
\dlim_{t\downarrow 0}u(0,t) = v_{0,-}, \quad\quad \dlim_{t\downarrow 0}u(1,t) = v_{0,+}.
\end{equation}
The reason for the requirement \eqref{ep3} is that the stiky Brownian motion spends a positive amount of time at $0$ and $1$ with positive probability.

Moreover, suppose that at each boundary point $0$ and $1$, there is a reservoir of mass $v_{\pm}(t)\in C[0,\infty)$ changing in time. The boundary conditions \eqref{ep2} describe that the mass flux at $0$ and $1$ equals the mass change in the left and right reservoir respectively. If we identify $v_-(t)=u(0,t)$ and $v_+(t)=u(1,t)$ then it can be checked that
$$ \dint_0^1u(r,t)\,dr+v_-(t)+v_+(t)=\dint_0^1u_0(r)\,dr+v_{0,-}+v_{0,+}, \forall t>0. $$
This implies the conservation of mass.

Now in the current paper, we will examine an analogous boundary value problem but depending on a parameter $\s>0$. The role of $\s$ will be explained later. For any fixed $\s>0$, let us consider the heat equation with the initial datum $u^{\s}_0\in C(0,1)$ taking values in $[0,1]$ as follows
\begin{equation}\label{heat}
u^{\s}_t=\dfrac{1}{2}u^{\s}_{rr},\quad \dlim_{t\downarrow 0} u^{\s}(r,t)=u_0^{\s}(r).
\end{equation}

Let us impose a reservoir of mass $\rho^{\s}_{\pm}(t)\in C[0,\infty)$ at each boundary. We observe that if $u^{\s}(0,\cdot)$ is larger than $\rho^{\s}_-(\cdot)$, the rate of the mass change in the left reservoir increases in time. Moreover, the larger this difference is, the faster the rate increases. We will use the factor $\s^{-1}$ to emphasize this property. Although $\rho^{\s}_-(\cdot)$ may be different from $u^{\s}(0,\cdot)$, the difference between them becomes $0$ as $\s$ goes to $0$. An analogous phenomenon also happens at the boundary point $1$. Then for a parameter $\s>0$, this fact can be described in a rigorous way
\begin{align}
\dfrac{d}{dt}\rho^{\s}_-(t)&=\s^{-1}\big[u^{\s}(0,t)-\rho^{\s}_-(t)\big]\label{dif1}\\
\dfrac{d}{dt}\rho^{\s}_+(t)&=\s^{-1}\big[u^{\s}(1,t)-\rho^{\s}_+(t)\big]\label{dif2}.
\end{align}

Furthermore, since we will look for a solution in $C^{2,1}((0,1)\times (0,\infty))$, the boundary conditions that implies the conservation of mass are expressed in a weak form
\begin{align}
\rho^{\s}_-(t)&=\rho^{\s}_-(0)+\dlim_{l\to 0}\dlim_{t_0\downarrow 0}\dint_{t_0}^t\dfrac{1}{2}u^{\s}_r(l,s)\,ds\label{dif3}\\
\rho^{\s}_+(t)&=\rho^{\s}_+(0)-\dlim_{l\to 1}\dlim_{t_0\downarrow 0}\dint_{t_0}^t\dfrac{1}{2}u^{\s}_r(l,s)\,ds\label{dif4}.
\end{align}

For a fixed $\s>0$, if we require that $u^{\s}(0,\cdot), \,u^{\s}(1,\cdot): [0,\infty)\to [0,1]$ and $u^{\s}(0,\cdot), \,u^{\s}(1,\cdot)\,\in C(0,\infty)$, then there exist unique solutions $u^{\s}\in C^{2,1}((0,1)\times (0,\infty)), \rho^{\s}_{\pm}\in C^1(0,\infty)\cap C([0,\infty))$ to the boundary value problem \eqref{heat}-\eqref{dif4}. The first part of the current paper is devoted to arrive at this result. Next, by the suitably chosen initial data, identifying the unique limits of these solutions as $\s\to 0$ and investigating their regularities lead us to the existence and uniqueness of the solution $u\in C^{2,1}([0,1]\times (0,\infty))$ satisfying \eqref{ep1}-\eqref{ep3}.

Moreover, in the second part of this paper, based on the fact that a sticky Brownian motion on the half line $[0, \infty)$ solves a stochastic differential equation as verified in \cite{EP}, we will give an analogous characterization for a sticky Brownian motion on the bounded interval $[0,1]$. This allows us to represent the unique solution to the two-point boundary value problem \eqref{ep1}-\eqref{ep3} probabilistically in terms of a sticky Brownian motion by applying Ito's formula.
\begin{rem}\label{resc}
In \cite{N}, we consider a system of particles moving according to the simple symmetric exclusion process in the channel $[1,N]$ with reservoirs at the boundaries. The reservoirs of size $N$ are also particle systems which can be exchanged with the ones in the channel. The hydrodynamic limit equation we obtain for this particle system is the two-point boundary value problem mentioned above. From the physical point of view, the unique solution to this problem is the limit of a sequence of the one-body correlation functions for an appropriately constructed interacting particle system. Furthermore, by duality technique, one can also express the correlation function in terms of a sticky random walk. Since the convergence of a sequence of rescaled sticky random walks to a sticky Brownian motion can be shown based on the arguments presented in \cite{A}, it leads us to a probabilistic representation of this unique solution.
\end{rem}
\section{A two-point boundary value problem}
Let us set $U:=C^{2,1}((0,1)\times (0,\infty))$ and $H:=C^1(0,\infty)\cap C([0,\infty))$.
\begin{thm}\label{pde}
For any fixed $\s>0$, let $u^{\s}_0\in C(0,1)$ with values in $[0,1]$ and $\rho^{\s}_{\pm}(0)=v^{\s}_{0,\pm}\in [0,1]$. If $u^{\s}(0,\cdot), \,u^{\s}(1,\cdot): [0,\infty)\to [0,1]$ and $u^{\s}(0,\cdot), \,u^{\s}(1,\cdot)\,\in C(0,\infty)$, then there exists a unique $(u^{\s}, \rho^{\s}_-,\rho^{\s}_+)\in U\times H\times H$ which satisfies the following problem 
\begin{align}
u^{\s}_t&=\dfrac{1}{2}u^{\s}_{rr},\quad \dlim_{t\downarrow 0} u^{\s}(r,t)=u_0^{\s}(r)\label{h}\\
\dfrac{d}{dt}\rho^{\s}_-(t)&=\s^{-1}\big[u^{\s}(0,t)-\rho^{\s}_-(t)\big]\label{erv1}\\
\dfrac{d}{dt}\rho^{\s}_+(t)&=\s^{-1}\big[u^{\s}(1,t)-\rho^{\s}_+(t)\big]\label{erv2}\\
\rho^{\s}_-(t)&=\rho^{\s}_-(0)+\dlim_{l\to 0}\dlim_{t_0\downarrow 0}\dint_{t_0}^t\dfrac{1}{2}u^{\s}_r(l,s)\,ds\label{b1}\\
\rho^{\s}_+(t)&=\rho^{\s}_+(0)-\dlim_{l\to 1}\dlim_{t_0\downarrow 0}\dint_{t_0}^t\dfrac{1}{2}u^{\s}_r(l,s)\,ds.\label{b2}
\end{align}
\end{thm}

\vskip 0.3cm
The existence and uniqueness of $(u^{\s}, \rho^{\s}_{\pm})\in U\times H\times H$ can be shown by applying the same technique as presented in Proposition 3.10, \cite{N}, where we aim to arrive at an integral equation and then construct its unique solution inductively by using the contraction mapping theorem.

\begin{proof}
For $(r,t)\in [0,1]\times (0,\infty)$, we denote
$$ \Theta(r,t)=\dsum_{n=-\infty}^{+\infty}\dfrac{1}{\sqrt{2\pi t}}e^{-\frac{(r+2n)^2}{2t}}. $$

Then as a result of Theorem 6.3.1, \cite{C}, for any fixed $\s>0$, $(r,t)\in (0,1)\times (0,\infty)$, the function $u^{\s}\in C^{2,1}((0,1)\times (0,\infty))$ defined by the following expression
\begin{align}\label{u}
 u^{\s}(r,t)=&\dint_0^1u^{\s}_0(r')[\Theta(r-r',t)-\Theta(r+r',t)]\,dr'\notag\\
&-\dint_0^t\dfrac{\partial \Theta}{\partial r}(r,t-s)u^{\s}(0,s)\,ds+\dint_0^t\dfrac{\partial \Theta}{\partial r}(r-1,t-s)u^{\s}(1,s)\,ds
\end{align}
satisfies the linear heat equation \eqref{h} with boundary values $u^{\s}(0,\cdot), \,u^{\s}(1,\cdot)$ and the initial datum $u^{\s}_0$. 

Thus \eqref{b1} and \eqref{b2} give us
\begin{align}
\rho^{\s}_-(t)=v^{\s}_{0,-}+&\dlim_{l\to 0}\dlim_{t_0\downarrow 0}\dint_{t_0}^t\dfrac{1}{2}\dint_0^1u^{\s}_0(r')\Bigg[\dfrac{\partial\Theta}{\partial r}(l-r',s)-\dfrac{\partial\Theta}{\partial r}(l+r',s)\Bigg]dr'ds\notag\\
&-\dint_0^t\Theta(0,t-s)u^{\s}(0,s)\,ds+\dint_0^t\Theta(1,t-s)u^{\s}(1,s)\,ds,\label{bc1}
\end{align}
and similarly,
\begin{align}
\rho^{\s}_+(t)=v^{\s}_{0,+}-&\dlim_{l\to 1}\dlim_{t_0\downarrow 0}\dint_{t_0}^t\dfrac{1}{2}\dint_0^1u^{\s}_0(r')\Bigg[\dfrac{\partial\Theta}{\partial r}(l-r',s)-\dfrac{\partial\Theta}{\partial r}(l+r',s)\Bigg]dr'ds\notag\\
&+\dint_0^t\Theta(1,t-s)u^{\s}(0,s)\,ds-\dint_0^t\Theta(0,t-s)u^{\s}(1,s)\,ds.\label{bc2}
\end{align}

On the other hand, since it follows from \eqref{erv1} and \eqref{erv2} that
\begin{align}
\rho^{\s}_-(t)&=e^{-\s^{-1}(t-t_0)}\rho^{\s}_-(t_0)+\dint_{t_0}^t\s^{-1}e^{-\s^{-1}(t-s)}u^{\s}(0,s)\,ds\label{erv3}\\
\rho^{\s}_+(t)&=e^{-\s^{-1}(t-t_0)}\rho^{\s}_+(t_0)+\dint_{t_0}^t\s^{-1}e^{-\s^{-1}(t-s)}u^{\s}(1,s)\,ds\label{erv4}
\end{align}
then we obtain the following system
\begin{equation}\label{sv}
\begin{cases}
&\dint_0^t\Big[\Theta(0,t-s)+\s^{-1}e^{-\s^{-1}(t-s)}\Big]u^{\s}(0,s)\,ds+\dint_0^t-\Theta(1,t-s)u^{\s}(1,s)\,ds\\
&\hskip6.5cm=f^{\s}_-(t)+v_{0,-}(1-e^{-\s^{-1}t})\\
&\dint_0^t-\Theta(1,t-s)u^{\s}(0,s)\,ds+\dint_0^t\Big[\Theta(0,t-s)+\s^{-1}e^{-\s^{-1}(t-s)}\Big]u^{\s}(1,s)\,ds\\
&\hskip6.5cm=f^{\s}_+(t)+v_{0,+}(1-e^{-\s^{-1}t}),
\end{cases}
\end{equation}
where
$$f^{\s}_-(t)=\dlim_{l\to 0}\dlim_{t_0\downarrow 0}\dint_{t_0}^t\dfrac{1}{2}\dint_0^1u^{\s}_0(r')\bigg[\dfrac{\partial\Theta}{\partial r}(l-r',s)-\dfrac{\partial\Theta}{\partial r}(l+r',s)\bigg]dr'ds,$$
$$f^{\s}_+(t)=-\dlim_{l\to 1}\dlim_{t_0\downarrow 0}\dint_{t_0}^t\dfrac{1}{2}\dint_0^1u^{\s}_0(r')\bigg[\dfrac{\partial\Theta}{\partial r}(l-r',s)-\dfrac{\partial\Theta}{\partial r}(l+r',s)\bigg]dr'ds.$$

Now multiplying both sides of the first equation of the above system by $(x-t)^{-1/2}$ and integrating with respect to $t$ from $0$ to $x$ yield 
\begin{align}\label{sv1}
&\dint_0^x\sqrt{\dfrac{\pi}{2}}u^{\s}(0,s)\,ds\notag\\
+&\dint_0^x\bigg[\dint_0^1\dsum_{n\geq 1}\dfrac{2}{\sqrt{2\pi}}\dfrac{1}{\sqrt{y(1-y)}}e^{-\frac{(2n)^2}{2y(x-s)}}\,dy+\dint_s^x\dfrac{\s^{-1}}{\sqrt{x-t}}e^{-\s^{-1}(t-s)}\,dt\bigg]u^{\s}(0,s)\,ds\notag\\
&\hskip 3cm+\dint_0^x\bigg[\dint_0^1-\dsum_{n\geq 1}\dfrac{2}{\sqrt{2\pi}}\dfrac{1}{\sqrt{y(1-y)}}e^{-\frac{(2n-1)^2}{2y(x-s)}}\,dy\bigg]u^{\s}(1,s)\,ds\notag\\
&\hskip 4.5cm=\dint_0^x\dfrac{1}{\sqrt{x-t}}\Big(f^{\s}_-(t)+v_{0,-}(1-e^{-\s^{-1}t})\Big)\,dt.
\end{align}

If $u^{\s}(0,\cdot), \,u^{\s}(1,\cdot)\,\in C(0,\infty)$ and take values in $[0,1]$ on $[0,\infty)$ then this enables us to define for $x>0$, 
$$ \psi^{\s}_-(x)=\dint_0^xu^{\s}(0,s)\,ds,\quad\quad\quad \psi^{\s}_+(x)=\dint_0^xu^{\s}(1,s)\,ds $$
and thus
$$ \dfrac{d}{dx}\psi^{\s}_-(x)=u^{\s}(0,x), \quad\quad\quad \dfrac{d}{dx}\psi^{\s}_+(x)=u^{\s}(1,x). $$

For $\alpha:=\sqrt{2/\pi}$, applying integration by parts in \eqref{sv1} gives us
\begin{align*}
\psi^{\s}_-(x)=&\dint_0^x\alpha\bigg[\dint_0^1\dsum_{n\geq 1}\dfrac{-2}{\sqrt{2\pi}}\dfrac{1}{\sqrt{y(1-y)}}e^{-\frac{(2n)^2}{2y(x-s)}}\dfrac{(2n)^2}{2y(x-s)^2}\,dy\\
&\hskip 0.5cm-\dfrac{\s^{-1}}{\sqrt{x-s}}+2\s^{-2}\sqrt{x-s}-\dint_s^x2\s^{-3}\sqrt{x-t}\,e^{-\s^{-1}(t-s)}\,dt\bigg]\psi^{\s}_-(s)\,ds\\
&+\dint_0^x\bigg[\dint_0^1\dsum_{n\geq 1}\dfrac{2\alpha}{\sqrt{2\pi}}\dfrac{1}{\sqrt{y(1-y)}}e^{-\frac{(2n-1)^2}{2y(x-s)}}\dfrac{(2n-1)^2}{2y(x-s)^2}\,dy\bigg]\psi^{\s}_+(s)\,ds\\
&+\dint_0^x\dfrac{\alpha}{\sqrt{x-t}}\Big(f^{\s}_-(t)+v_{0,-}(1-e^{-\s^{-1}t})\Big)\,dt.
\end{align*}

Making a similar argument as above for the second equation of the system \eqref{sv} leads us to consider the equation
\begin{equation}\label{Vol}
\begin{bmatrix}
\psi^{\s}_-(x)\\ 
\psi^{\s}_+(x)\\ 
\end{bmatrix}=\begin{bmatrix}
F^{\s}_-(x)\\ 
F^{\s}_+(x)\\ 
\end{bmatrix}+\dint_0^x \begin{bmatrix}
K^{\s}_-(x-s)&K^{\s}_+(x-s)\\ 
K^{\s}_+(x-s)&K^{\s}_-(x-s)\\ 
\end{bmatrix}\begin{bmatrix}
\psi^{\s}_-(s)\\ 
\psi^{\s}_+(s)\\ 
\end{bmatrix}ds,
\end{equation}
where 
\begin{align*}
K^{\s}_-(t)&=\alpha\bigg[\dint_0^1\dsum_{n\geq 1}\dfrac{-2}{\sqrt{2\pi}}\dfrac{1}{\sqrt{y(1-y)}}e^{-\frac{(2n)^2}{2yt}}\dfrac{(2n)^2}{2yt^2}\,dy\\
&\hskip3cm-\dfrac{\s^{-1}}{\sqrt{t}}+2\s^{-2}\sqrt{t}-\dint_0^t2\s^{-3}\sqrt{t-\sigma}\,e^{-\s^{-1}\sigma}\,d\sigma\bigg],\\
K^{\s}_+(t)&=\dint_0^1\dsum_{n\geq 1}\dfrac{-2\alpha}{\sqrt{2\pi}}\dfrac{1}{\sqrt{y(1-y)}}e^{-\frac{(2n-1)^2}{2yt}}\dfrac{(2n-1)^2}{2yt^2}\,dy,\\
F^{\s}_{\pm}(x)&=\dint_0^x\dfrac{\alpha}{\sqrt{x-t}}\Big(f^{\s}_{\pm}(t)+v_{0,\pm}(1-e^{-\s^{-1}t})\Big)\,dt.
\end{align*}

Applying the same technique as introduced in Proposition 3.10, \cite{N}, one can verify the following result.
\begin{prop}\label{Vol1}
The equation \eqref{Vol} has a unique solution $(\psi^{\s}_-,\psi^{\s}_+)\in C(0,T]\times C(0,T]$ for any $T>0$.
\end{prop}

As a consequence of the above result, there exists a unique solution $(u^{\s}(0,\cdot), u^{\s}(1,\cdot))$ to the system \eqref{sv} for any $\s>0$ fixed. Then by the expressions \eqref{u}, \eqref{erv3} and \eqref{erv4}, we obtain the unique existence of the solution $(u^{\s}, \rho^{\s}_-,\rho^{\s}_+)\in U\times H\times H$ to our main problem.
\end{proof}

Let us now identify the limits of sequences of functions $u^{\s}, \rho^{\s}_{\pm}$ as $\s$ goes to $0$ up to subsequences. Since the uniqueness of the limits can be verified, we obtain the identification of the limits $u, v_{\pm}$ for the original sequences. Moreover, the boundary conditions can be attained in a strong form in view of the continuous differentiability of $v_{\pm}$. More precisely, the limit $u\in C^{2,1}([0,1]\times (0,\infty))$ is the unique solution to a two-point boundary value problem.
\begin{thm}\label{bvp}
Let $u_0\in C(0,1)$ with values in $[0,1]$ and $v_{0,\pm}\in [0,1]$. There exists a unique $u\in C^{2,1}([0,1]\times (0,\infty))$ which solves the following problem
\begin{equation}\label{lu1}
 u_t =  \frac 1{2} u_{rr},  \quad u(1,t)=v_+(t),  u(0,t)=v_-(t),\quad \lim_{t\downarrow 0}u(r,t) =u_0(r)
 \end{equation}
with $v_{\pm}(t)$ such that for any $t>0$,
\begin{equation}\label{sbc}
\dfrac{d}{dt}v_-(t)=\dfrac{1}{2}u_r(0,t), \quad\quad \dfrac{d}{dt}v_+(t)=-\dfrac{1}{2}u_r(1,t), \quad\quad \dlim_{t\downarrow 0}v_{\pm}(t) = v_{0,\pm}.
\end{equation}
\end{thm}
\begin{proof}
For any fixed $\s>0$, let $(u^{\s}, \rho^{\s}_{\pm})$ be the unique solution obtained in Theorem \ref{pde}. At the initial time, the sequence $u^{\s}_0$ is chosen such that it converges uniformly to $u_0$ on any compact set of $(0,1)$ as $\s\to 0$. Moreover, let us select the sequences $v^{\s}_{0,\pm}$ which converge to $ v_{0,\pm}$, respectively, as $\s\to 0$.

We observe that from \eqref{bc1} and \eqref{bc2}, for any $\delta, T>0$, there exists a constant $C$ such that for any $s,t\in [\delta, T]$,
$$ |\rho^{\s}_{\pm}(t)-\rho^{\s}_{\pm}(s)|\leq C|t-s|, $$
thus the sequences $(\rho^{\s}_{\pm})_{\s}$ are uniformly equicontinuous. Moreover, \eqref{erv3} and \eqref{erv4} imply the uniform boundedness of these sequences on $[\delta, T]$. Therefore, there exist subsequences $\rho^{\s_k}_{\pm}$ converge uniformly on $[\delta, T]$ to $v_{\pm}$, respectively. Notice that $v_{\pm}\in C(0,\infty)$.

On the other hand, by \eqref{erv3}, we claim that for any $t>0$, 
$$\dlim_{\s\to 0}\big|\rho^{\s}_-(t)-u^{\s}(0,t)\big|=0.$$

Indeed, since $u^{\s}(0,\cdot)$ is continuous at $t>0$ then there exists $\delta'>0$ depending on $\s$ and $t$ such that $\big|u^{\s}(0,s)-u^{\s}(0,t)\big|<\s$ for any $s\in [t-\delta', t]$. Therefore, our claim follows from the following estimate
\begin{align*}
&\big|\rho^{\s}_-(t)-u^{\s}(0,t)\big|\\
\leq& e^{-\s^{-1}t}\big|v_{0,\pm}-u^{\s}(0,t)\big|+\dint_0^t\s^{-1}e^{-\s^{-1}(t-s)}\big|u^{\s}(0,s)-u^{\s}(0,t)\big|\,ds\\
\leq& e^{-\s^{-1}t}+\dint_0^{t-\delta'}\s^{-1}e^{-\s^{-1}(t-s)}\,ds+\s\dint_{t-\delta'}^t\s^{-1}e^{-\s^{-1}(t-s)}\,ds.
\end{align*}

Analogously, it can be obtained from \eqref{erv4} that for any $t>0$, 
$$\dlim_{\s\to 0}\big|\rho^{\s}_+(t)-u^{\s}(1,t)\big|=0,$$
then the subsequences $u^{\s_k}(0,\cdot), u^{\s_k}(1,\cdot)$ converge pointwise on $(0,\infty)$ to $v_{\mp}$, respectively.

Hence, for any $(r,t)\in (0,1)\times (0,\infty)$, the corresponding subsequence $u^{\s_k}$ converges to $u$ given by
\begin{align}\label{u1}
 u(r,t)=&\dint_0^1u_0(r')[\Theta(r-r',t)-\Theta(r+r',t)]\,dr'\notag\\
&-\dint_0^t\dfrac{\partial \Theta}{\partial r}(r,t-s)v_-(s)\,ds+\dint_0^t\dfrac{\partial \Theta}{\partial r}(r-1,t-s)v_+(s)\,ds.
\end{align}

One can easily checked that $u\in C^{2,1}((0,1)\times (0,\infty))$. As already shown in Chapter 6 of \cite{C}, the limit $u$ solves the linear heat equation with boundary values $v_{\pm}$ and the initial datum $u_0$. Moreover, we observe that
\begin{align*}
\dlim_{t\downarrow 0}v_{\pm}(t)&=\dlim_{t\downarrow 0}\dlim_{k\to\infty}\rho^{\s_k}_{\pm}(t)=\dlim_{k\to\infty}\dlim_{t\downarrow 0}\rho^{\s_k}_{\pm}(t)\\
&=\dlim_{k\to\infty}\rho^{\s_k}_{\pm}(0)=\dlim_{k\to\infty}v^{\s_k}_{0,\pm}=v_{0,\pm}.
\end{align*}

Taking the limit of both sides of \eqref{bc1} and \eqref{bc2} along the subsequences $\rho^{\s_k}_{\pm}$ with our choice of the sequences $u^{\s}_0, v^{\s}_{0,\pm}$ yields
\begin{equation}\label{lubc}
\begin{cases}
v_-(t)&=v_{0,-}+\dlim_{l\to 0}\dlim_{t_0\downarrow 0}\dint_{t_0}^t \dfrac{1}{2}u_r(l,s)\,ds\\
v_+(t)&=v_{0,+}-\dlim_{l\to 1}\dlim_{t_0\downarrow 0}\dint_{t_0}^t \dfrac{1}{2}u_r(l,s)\,ds.
\end{cases}
\end{equation}

More precisely, $u$ satisfies \eqref{lu1} with $v_{\pm}(t)$ such that $\dlim_{t\downarrow 0}v_{\pm}(t) = v_{0,\pm}$ and the boundary conditions \eqref{lubc}.

So far we have just identified the limit of the sequence $u^{\s}$ up to a subsequence. Let us now consider other subsequences $\rho^{\s_m}_{\pm}$ such that they converge uniformly to $\hat{v}_{\pm}$, respectively. Then the corresponding limit $\hat{u}$ of the subsequence $u^{\s_m}$ can be given by the same expression as in \eqref{u1}, where $v_{\pm}$ are replaced by $\hat{v}_{\pm}$. Applying the same argument as before, we deduce that $\hat{u}$ also solves the problem \eqref{lu1} with the boundary conditions \eqref{lubc}, where we replace $v_{\pm}$ by $\hat{v}_{\pm}$. 

We denote $\bar{v}_{\pm}=v_{\pm}-\hat{v}_{\pm}$. The boundary conditions \eqref{lubc} imply that
\begin{align}
v_-(t)=v_{0,-}+&\dlim_{l\to 0}\dlim_{t_0\downarrow 0}\dint_{t_0}^t\dfrac{1}{2}\dint_0^1u_0(r')\bigg[\dfrac{\partial\Theta}{\partial r}(l-r',s)-\dfrac{\partial\Theta}{\partial r}(l+r',s)\bigg]dr'ds\notag\\
&-\dint_0^t\Theta(0,t-s)v_-(s)\,ds+\dint_0^t\Theta(1,t-s)v_+(s)\,ds,\label{first}
\end{align}
and similarly,
\begin{align}
v_+(t)=v_{0,+}-&\dlim_{l\to 1}\dlim_{t_0\downarrow 0}\dint_{t_0}^t\dfrac{1}{2}\dint_0^1u_0(r')\bigg[\dfrac{\partial\Theta}{\partial r}(l-r',s)-\dfrac{\partial\Theta}{\partial r}(l+r',s)\bigg]dr'ds\notag\\
&+\dint_0^t\Theta(1,t-s)v_-(s)\,ds-\dint_0^t\Theta(0,t-s)v_+(s)\,ds.\label{second}
\end{align}

This implies
$$\bar{v}_-(t)=-\dint_0^t\Theta(0,t-s)\bar{v}_-(s)\,ds+\dint_0^t\Theta(1,t-s)\bar{v}_+(s)\,ds,$$
and similarly,
$$ \bar{v}_+(t)= \dint_0^t\Theta(1,t-s)\bar{v}_-(s)\,ds-\dint_0^t\Theta(0,t-s)\bar{v}_+(s)\,ds. $$

By setting $\mathbf{V}(t)=\bar{v}_-(t)+\bar{v}_+(t)$, the two above expressions allow us to attain that
$$ \mathbf{V}(t)=\dint_0^t\big(\Theta(1,t-s)-\Theta(0,t-s)\big)\mathbf{V}(s)\,ds. $$

This gives us $\mathbf{V}(t)=0, \forall t> 0$, by applying Gronwall's inequality. Since for any $t>0$, $\bar{v}_{\pm}(t)\in[0,1]$, then $v_{\pm}(t)=\hat{v}_{\pm}(t)$ and $u(r,t)=\hat{u}(r,t)$ for all $(r,t)\in [0,1]\times (0,\infty)$. This leads us to the uniqueness of the solution to the problelm \eqref{lu1} with the boundary conditions \eqref{lubc}. 

Hence, we have verified that $u$ is the pointwise limit of the sequence $u^{\s}$ for $(r,t)\in (0,1)\times (0,\infty)$ as $\s$ goes to $0$.

Moreover, in Section 3.6, \cite{N}, it can be shown that $v_{\pm}\in C^1(0,\infty)$. Thus $u\in C^{2,1}([0,1]\times (0,\infty))$ and now one can rewrite the boundary conditions \eqref{lubc} in the strong form \eqref{sbc}. It completes the proof of Theorem \ref{bvp}.
\end{proof}
\section{Sticky Brownian motion}
\subsection{Sticky Brownian motion as a strong limit of a sequence of rescaled sticky random walks}
Sticky random walk $(X(t))_{t\geq 0}$ moving on $[0,N+1]\cap \mathbb{N}$ is a continuous time random walk with jump rates $c(x,x\pm 1)=\dfrac{1}{2}, \forall x\in [1,N]\cap \mathbb{N}$ and $c(0,1)=c(N+1,N)=\dfrac{1}{2N}$.

Let $Y$ be a simple symmetric random walk on $\mathbb{Z}$ starting from $x$. Recall that the sequence of rescaled random walks $N^{-1}Y(N^2t)$ converges uniformly almost surely on compact intervals of $[0,\infty)$ to a Brownian motion $B, B_0=r\in [0,1]$, defined on some rich enough common probability space $(\tilde{\Omega}, \mathcal{F}, P)$, see \cite{K}.

We denote by $Y^{\rm rf}$ the simple random walk $Y$ reflected at $0$ and $N+1$. Let us call
  \begin{equation*}
 \mathbf{T}(0,N+1;t;Y^{\rm rf}) = \int_0^t \big( \mathbf 1_{Y^{\rm rf}(s)=0}+ \mathbf 1_{Y^{\rm rf}(s)=N+1}\big)\,ds
 \end{equation*}
 the local time spent by $Y^{\rm rf}$ at $0$ and $N+1$. Then it is shown in \cite{N} that the sticky random walk $X$ can be realized by setting
 \begin{equation}\label{realize}
X\Big(t+ (2N-1)\mathbf{T}(0,N+1;t;Y^{\rm rf})\Big)=Y^{\rm rf}(t).
 \end{equation}
\begin{thm}\label{loti}
$(2N-1)N^{-2} \mathbf{T}(0,N+1;N^2t;Y^{\rm rf})$ converges uniformly almost surely on compact intervals of $[0,\infty)$ to $L_t$ which is the local time at $0$ and $1$ of the reflecting Brownian motion $B^{\rm rf}$ on $[0,1]$. Moreover, the rescaled sticky random walk $N^{-1} X(N^2t)$ converges uniformly almost surely on compact intervals of $[0,\infty)$ to the sticky Brownian motion $B^{\rm st}$ on $[0,1]$ defined as
  \begin{equation}
 \label{2.15}
B^{\rm st} (t+ L_t )= B^{\rm rf}(t).
 \end{equation}
\end{thm}
\begin{proof}

We know that the continuous time random walk $Y$ can be defined by $Y(t)=S_{\mathcal{N}(t)}$, where $S$ is a simple symmetric discrete time random walk and $\mathcal{N}$ is a Poisson process of parameter $1$. Let us denote by $\mu_k^{(m)}$ the number of visits to $m\in\mathbb{Z}$ in the first $k$ steps of the random walk $S$. Then for $\mu_k:=\dsum_{m\in\mathbb{Z}}\mu_k^{(m(N+1))}$, we can write
\begin{align*}
\dfrac{2N-1}{N^2}\mathbf{T}(0,N+1;N^2t;Y^{\rm rf})=&\dfrac{2N-1}{N^2}\dint_0^{N^2t}\big(\mathbf{1}_{Y^{\rm rf}(s)=0}+\mathbf{1}_{Y^{\rm rf}(s)=N+1}\big)\,ds\\
=&\dfrac{2N-1}{N^2}\dint_0^{N^2t}\mathbf{1}_{(N+1)^{-1}Y(s)\in\mathbb{Z}}\,ds\\
=&\dfrac{2N-1}{N^2}\dsum_{k=1}^{\mu_{\mathcal{N}(N^2t)}} G_k,
\end{align*}
where $G_k$ are independent exponential random variables with parameter $1$. We next verify the following result.
\begin{prop}
For any $T>0$ and $m\in\mathbb{Z}$,
$$ \sup_{t\in[0,T]}\bigg|\dfrac{2N-1}{N^2}\dsum_{k=1}^{\mu_{\mathcal{N}(N^2t)}}G_k-\dfrac{2}{N}\mu_{\mathcal{N}(N^2t)}\bigg|\overset{a.s.}\longrightarrow 0. $$
\end{prop}
\begin{proof}
As a consequence of the Borel-Cantelli lemma, it is enough to verify that for any $T>0$ and $\varepsilon>0$,
\begin{equation}\label{GN}
\dsum_{N\rightarrow\infty}P\bigg(\sup_{t\in[0,T]}\bigg|\dfrac{2N-1}{N^2}\dsum_{k=1}^{\mu^{(m(N+1))}_{\mathcal{N}(N^2t)}}G_k-\dfrac{2}{N}\mu^{(m(N+1))}_{\mathcal{N}(N^2t)}\bigg|>2\varepsilon\bigg)<\infty.
\end{equation}

This follows from applying Doob's martingale inequality and Markov's inequality.
\end{proof}

On the other hand, making use of the same arguments as in \cite{A} leads us to the fact that for any $T>0$ and $m\in\mathbb{Z}$,
$$ \sup_{t\in[0,T]}\bigg|\dfrac{2}{N}\mu_{[N^2t]}^{(m(N+1))}-2L_t^m(B)\bigg|\overset{a.s.}\longrightarrow 0, $$
where $L_t^m(B)$ stands for the local time at $m$ of the Brownian motion $B$. It follows that
$$ \sup_{t\in[0,T]}\bigg|\dfrac{2}{N}\mu_{\mathcal{N}(N^2t)}-2\dsum_{m\in\mathbb{Z}}L_t^m(B)\bigg|\overset{a.s.}\longrightarrow 0. $$

Combining with the above proposition, we can conclude that
$$ \sup_{t\in[0,T]}\bigg|\dfrac{2N-1}{N^2}\dsum_{k=1}^{\mu_{\mathcal{N}(N^2t)}}G_k-2\dsum_{m\in\mathbb{Z}}L_t^m(B)\bigg|\overset{a.s.}\longrightarrow 0. $$

Since it can be checked that
$$ 2\dsum_{m\in\mathbb{Z}}L_t^m(B)=2\dsum_{m\in\mathbb{Z}}L_t^{2m}(B)+2\dsum_{m\in\mathbb{Z}}L_t^{2m+1}(B)=L_t^0(B^{\rm rf})+L_t^1(B^{\rm rf})=L_t, $$
then for any $T>0$, this implies almost surely that 
$$\dlim_{N\to\infty}\sup_{t\in [0,T]}\bigg|\dfrac{2N-1}{N^2}\mathbf{T}(0,N+1;N^2t;Y^{\rm rf})-L_t\bigg|=0.$$

For any $T>0$, applying again the same arguments as in \cite{A} yields
$$ P(\dlim_{N\to\infty}\sup_{t\in [0,T]}\big|N^{-1}X(N^2t)-B^{\rm {st}}(t)|=0)=1. $$
\end{proof}
\subsection{Sticky Brownian motion as a solution to a stochastic differential equation}
The sticky Brownian motion on the bounded interval $[0,1]$ also solves a stochastic differential equation similar to the one on the half line $[0,\infty)$ as considered in \cite{EP}.
\begin{prop}\label{RWY}
The sticky Brownian motion $B^{\rm st}, B^{\rm st}(0)=r$, defined in \eqref{2.15} and the unique solution $\mathbf{B}$ to the following stochastic differential equation
\begin{equation}\label{SDE}
\begin{cases}
d\mathbf{B}(t)&=\mathbf{1}_{0<\mathbf{B}(t)<1}\,dW(t)+\dfrac{1}{2}\,\mathbf{1}_{\mathbf{B}(t)=0}\,dt-\dfrac{1}{2}\,\mathbf{1}_{\mathbf{B}(t)=1}\,dt,\\
\mathbf{B}(0)&=r,
\end{cases}
\end{equation}
for some standard Brownian motion $W$ and $r\in [0,1]$, have the same law.
\end{prop}

Besides the approach mentioned in Remark \ref{resc}, the above proposition gives us another way to attain the probabilistic representation of the unique solution to the problelm \eqref{lu1} with the boundary conditions \eqref{sbc}.
\begin{thm}\label{Ito}
The unique solution $u$ to the boundary value problelm \eqref{lu1}, \eqref{sbc} can be represented probabilistically, for $r\in[0,1], t>0$, as
\begin{equation}\label{eua}
u(r,t)=E_r[u_0(\mathbf{B}(t))\mathbf{1}_{0<\mathbf{B}(t)<1}+v_{0,-}\mathbf{1}_{\mathbf{B}(t)=0}+v_{0,+}\mathbf{1}_{\mathbf{B}(t)=1}], 
\end{equation}
where $\mathbf{B}$ solves the stochastic differential equation \eqref{SDE}.
\end{thm}

Proposition \ref{RWY} and Theorem \ref{Ito} are verified in the next sections.
\subsection{Proof of Proposition \ref{RWY}}
The existence and uniqueness in law of the solution to the stochastic differential equation \eqref{SDE} are proved in \cite{S}.

Therefore, it suffices to verify that the sticky Brownian motion $B^{\rm st}$ and the unique solution $\mathbf{B}$ have the same law. The idea to show this is looking for a Skorokhod problem that a suitable time change of the process $\mathbf{B}$ and the reflecting Brownian motion satisfy. More precisely, the proof consists of the following steps.

\begin{description}
\item[Step 1.] First, we show that the stochastic differential equation \eqref{SDE} is equivalent to the following system
\begin{equation}\label{SDE1}
\begin{cases}
d\mathbf{B}(t)&=\mathbf{1}_{0<\mathbf{B}(t)<1}\,dW(t)+\dfrac{1}{2}\,dL_t^0(\mathbf{B})-\dfrac{1}{2}\,dL_t^{1-}(\mathbf{B}),\\
dL_t^0(\mathbf{B})&=\mathbf{1}_{\mathbf{B}(t)=0}\,dt\\
dL_t^{1-}(\mathbf{B})&=\mathbf{1}_{\mathbf{B}(t)=1}\,dt\\
\mathbf{B}(0)&=r,
\end{cases}
\end{equation}
where $L_t^a(\mathbf{B})$ stands for the local time of $\mathbf{B}$ at $a$.

It is obvious that the system \eqref{SDE1} implies \eqref{SDE}. For the converse, we remark that $0\leq \mathbf{B}(t)\leq 1$ almost surely for any $t\geq 0$ if $\mathbf{B}$ is a solution of \eqref{SDE}. This follows from using the Ito - Tanaka formula (see Theorem 1.2, \cite{RY}), namely
\begin{align*}
\mathbf{B}(t)^-&=-\dint_0^t\mathbf{1}_{\mathbf{B}(s)<0}\,d\mathbf{B}(s)+\dfrac{1}{2}L_t^{0-}(\mathbf{B})=0,\\
(\mathbf{B}(t)-1)^+&=\dint_0^t\mathbf{1}_{\mathbf{B}(s)>1}\,d\mathbf{B}(s)+\dfrac{1}{2}L_t^1(\mathbf{B})=0,
\end{align*}
where we have used that
\begin{align*}
L_t^{0-}(\mathbf{B})&=\dlim_{\varepsilon\downarrow 0}\dfrac{1}{\varepsilon}\dint_0^t\mathbf{1}_{-\varepsilon\leq \mathbf{B}(s)<0}\,d[\mathbf{B}]_s\\
&=\dlim_{\varepsilon\downarrow 0}\dfrac{1}{\varepsilon}\dint_0^t\mathbf{1}_{-\varepsilon\leq \mathbf{B}(s)<0}\,\mathbf{1}_{0<\mathbf{B}(s)<1}\,ds=0,\\
L_t^1(\mathbf{B})&=\dlim_{\varepsilon\downarrow 0}\dfrac{1}{\varepsilon}\dint_0^t\mathbf{1}_{1\leq \mathbf{B}(s)<1+\varepsilon}\,d[\mathbf{B}]_s\\
&=\dlim_{\varepsilon\downarrow 0}\dfrac{1}{\varepsilon}\dint_0^t\mathbf{1}_{1\leq \mathbf{B}(s)<1+\varepsilon}\,\mathbf{1}_{0<\mathbf{B}(s)<1}\,ds=0.
\end{align*}

In view of this remark and using again the Ito - Tanaka formula, we get
\begin{align*}
\mathbf{B}(t)=\mathbf{B}(t)^+&=r+\dint_0^t\mathbf{1}_{\mathbf{B}(s)>0}\,d\mathbf{B}(s)+\dfrac{1}{2}L_t^0(\mathbf{B})\\
&=r+\dint_0^t\mathbf{1}_{0<\mathbf{B}(s)<1}\,dW(s)-\dint_0^t\dfrac{1}{2}\,\mathbf{1}_{\mathbf{B}(s)=1}\,ds+\dfrac{1}{2}L_t^0(\mathbf{B})
\end{align*}
and this implies $dL_t^0(\mathbf{B})=\mathbf{1}_{\mathbf{B}(t)=0}\,dt$.

Similarly, we have
\begin{align*}
-(\mathbf{B}(t)-1)&=(\mathbf{B}(t)-1)^-=-(r-1)-\dint_0^t\mathbf{1}_{\mathbf{B}(s)<1}\,d\mathbf{B}(s)+\dfrac{1}{2}L_t^{1-}(\mathbf{B})\\
&=-(r-1)-\dint_0^t\mathbf{1}_{0<\mathbf{B}(s)<1}\,dW(s)-\dint_0^t\dfrac{1}{2}\,\mathbf{1}_{\mathbf{B}(s)=0}\,ds+\dfrac{1}{2}L_t^{1-}(\mathbf{B})
\end{align*}
and this implies $dL_t^{1-}(\mathbf{B})=\mathbf{1}_{\mathbf{B}(t)=1}\,dt$. Hence, we obtain the equivalence of \eqref{SDE} and \eqref{SDE1}.
\item[Step 2.] Next, let us denote
$$ K(t)=\dint_0^t\mathbf{1}_{0<\mathbf{B}(s)<1}\,ds, \,\,\,\,\,\,\,\,\, \kappa(t)=\inf\{u\geq 0: K(u)>t\}. $$

Applying the time change $\kappa$ to the first equation of  \eqref{SDE1} yields
\begin{align*}
 V(t):=\mathbf{B}(\kappa(t))&=r+\dint_0^{\kappa(t)}\mathbf{1}_{0<\mathbf{B}(s)<1}\,dW(s)+\dfrac{1}{2}\,L_{\kappa(t)}^0(\mathbf{B})-\dfrac{1}{2}\,L_{\kappa(t)}^{1-}(\mathbf{B})\\
&=r+\dint_0^{\kappa(t)}\mathbf{1}_{0<\mathbf{B}(s)<1}\,dW(s)+\dfrac{1}{2}\,L_t^0( V)-\dfrac{1}{2}\,L_t^{1-}(V).
\end{align*}

Since $Q(t):=r+\dint_0^{\kappa(t)}\mathbf{1}_{0<\mathbf{B}(s)<1}\,dW(s)=r+\dint_0^t\mathbf{1}_{0<\mathbf{B}(\kappa(u))<1}\,dW(\kappa(u))$ is a continuous local martingale and note that
$$[Q]_t=\dint_0^t\mathbf{1}_{0<\mathbf{B}(\kappa(u))<1}\,d{\kappa(u)}=\dint_0^{\kappa(t)}\mathbf{1}_{0<\mathbf{B}(s)<1}\,ds=t,$$
then $Q$ is a Brownian motion starting from $r$ by P. Levy's characterization theorem. Moreover, due to the explicit expression of the solution to the Skorokhod problem
$$ V(t)=Q(t)+\dfrac{1}{2}\,L_t^0( V)-\dfrac{1}{2}\,L_t^{1-}(V), $$
(see, e.g., \cite{KLRS}, \cite{AM}, \cite{AAGP}), we obtain
\begin{equation}\label{f}
V(t)=\tilde{\mathbf{R}}_{0;1}(Q(t)), 
\end{equation}
where
$$ \tilde{\mathbf{R}}_{0;1}(Q(t)):=Q(t)-\big[(r-1)^+\wedge\inf_{u\in [0,t]}Q(u)\big]\vee\sup_{s\in[0,t]}\big[(Q(s)-1)\wedge\inf_{u\in [s,t]}Q(u)\big].$$
\item[Step 3.] Let us denote the Brownian motion $Q$ reflected at $0$ and $1$ by
\begin{equation}\label{rfl}
\mathbf{R}_{0;1}(Q(t)):=\dsum_{m\in\mathbb{Z}}|Q(t)-2m|\,\mathbf{1}_{|Q(t)-2m|\leq 1}.
\end{equation}

From \cite{G}, there is an explicit representation of $Z:=\mathbf{R}_{0;1}(Q)$, the Brownian motion $Q$ starting from $r$ reflected at two barriers $0$ and $1$, as follows
\begin{align*}
Z(t)=r+&\dint_0^t\big(\mathbf{1}_{Q(s)\in\underset{m\in\mathbb{Z}}\bigcup\,(2m,2m+1)}-\mathbf{1}_{Q(s)\in\underset{m\in\mathbb{Z}}\bigcup\,(2m+1,2m+2)}\big)\,dQ(s)\\
&+\dsum_{m\in\mathbb{Z}}L_t^{2m}(Q)-\dsum_{m\in\mathbb{Z}}L_t^{2m+1}(Q).
\end{align*}

Set $\hat{Q}(t):=r+\dint_0^t\big(\mathbf{1}_{\{Q(s)\in\underset{m\in\mathbb{Z}}\bigcup\,(2m,2m+1)\}}-\mathbf{1}_{\{Q(s)\in\underset{m\in\mathbb{Z}}\bigcup\,(2m+1,2m+2)\}}\big)\,dQ(s)$. We notice that $\hat{Q}$ is a Brownian motion starting from $r$ by P. Levy's characterization theorem. Moreover,
\begin{equation}\label{loc}
\dsum_{m\in\mathbb{Z}}L_t^{2m}(Q)=\dfrac{1}{2}L_t^0(Z)\,\,\,\text{ and }\,\,\,\dsum_{m\in\mathbb{Z}}L_t^{2m+1}(Q)=\dfrac{1}{2}L_t^{1-}(Z).
\end{equation}

Hence, we can write
$$ Z(t)=\hat{Q}(t)+ \dfrac{1}{2}L_t^0(Z)-\dfrac{1}{2}L_t^{1-}(Z).$$

Using again the explicit representation of the solution to the above Skorokhod problem gives us
\begin{equation}\label{g}
Z(t)=\tilde{\mathbf{R}}_{0;1}(\hat{Q}(t))\overset{d}=\tilde{\mathbf{R}}_{0;1}(Q(t)). 
\end{equation}
It follows from \eqref{f} and \eqref{g} that
\begin{equation}\label{Y}
 V(t)\overset{d}=Z(t)=\mathbf{R}_{0;1}(Q(t)).
\end{equation}
\item[Step 4.] Moreover, using the second and the third equation of the system \eqref{SDE1} gives us
\begin{align*}
\kappa(t)&=\dint_0^{\kappa(t)}\mathbf{1}_{0<\mathbf{B}(s)<1}\,ds+\dint_0^{\kappa(t)}\mathbf{1}_{\mathbf{B}(s)=0}\,ds+\dint_0^{\kappa(t)}\mathbf{1}_{\mathbf{B}(s)=1}\,ds\\
&=t+L_{\kappa(t)}^0(\mathbf{B})+L_{\kappa(t)}^{1-}(\mathbf{B})\\
&=t+L_t^0( V)+L_t^{1-}( V)\\
&\overset{d}=t+L_t^0(Z)+L_t^{1-}(Z).
\end{align*}

Then in view of \eqref{Y}, we deduce that 
$$Z(t)\overset{d}=\mathbf{B}(t+L_t^0(Z)+L_t^{1-}(Z)).$$

Since $Z(t)\overset{d}=B^{\rm rf}(t)$, the proof is complete.
\end{description}
\subsection{Proof of Theorem \ref{Ito}}
For any $\delta>0$, we fix $t_0\geq \delta$. Since the unique solution $u\in C^{2,1}([0,1]\times (0,\infty))$, we apply Ito's formula to the function $u(\mathbf{B}(t), t_0-t)$ for $t\in [0,t_0-\delta]$ and obtain that
\begin{align*}
&u(\mathbf{B}(t),t_0-t)\\
=&u(r,t_0)+\dint_0^tu_r(\mathbf{B}(s),t_0-s)\,d\mathbf{B}(s)+\dint_0^t\dfrac{1}{2}u_{rr}(\mathbf{B}(s),t_0-s)\mathbf{1}_{0<\mathbf{B}(s)<1}\,ds\\
&\hskip7.3cm+\dint_0^tu_s(\mathbf{B}(s),t_0-s)\,ds\\
=&u(r,t_0)+\dint_0^tu_r(\mathbf{B}(s),t_0-s)\mathbf{1}_{0<\mathbf{B}(s)<1}\,dW(s)\\
&\hskip2.5cm+\dint_0^t\Big[\dfrac{1}{2}u_r(\mathbf{B}(s),t_0-s)+\dfrac{d}{ds}v_-(t_0-s)\Big]\mathbf{1}_{\mathbf{B}(s)=0}\,ds\\
&\hskip2.5cm+\dint_0^t\Big[-\dfrac{1}{2}u_r(\mathbf{B}(s),t_0-s)+\dfrac{d}{ds}v_+(t_0-s)\Big]\mathbf{1}_{\mathbf{B}(s)=1}\,ds\\
=&u(r,t_0)+\dint_0^tu_r(\mathbf{B}(s),t_0-s)\mathbf{1}_{0<\mathbf{B}(s)<1}\,dW(s).
\end{align*}

Let us call $M(t):=\dint_0^tu_r(\mathbf{B}(s),t_0-s)\mathbf{1}_{0<\mathbf{B}(s)<1}\,dW(s)$. Then $M$ is a martingale since $u\in C^{2,1}([0,1]\times (0,\infty))$. So $E_r[M(t_0-\delta)]=E_r[M(0)]=0$. Hence,
$$u(r,t_0)=E_r[u(\mathbf{B}(t_0-\delta),\delta)].$$

Taking the limit $\delta\downarrow 0$ of both sides of the above equality gives us
\begin{align*}
u(r,t_0)=&\dlim_{\sigma\downarrow 0}E_r[u(\mathbf{B}(t_0),\sigma)]\\
=&\dlim_{\sigma\downarrow 0}E_r[u(\mathbf{B}(t_0),\sigma)\mathbf{1}_{0<\mathbf{B}(t_0)<1}+u(0,\sigma)\mathbf{1}_{\mathbf{B}(t_0)=0}+u(1,\sigma)\mathbf{1}_{\mathbf{B}(t_0)=1}]\\
=&E_r[u_0(\mathbf{B}(t_0))\mathbf{1}_{0<\mathbf{B}(t_0)<1}+v_{0,-}\mathbf{1}_{\mathbf{B}(t_0)=0}+v_{0,+}\mathbf{1}_{\mathbf{B}(t_0)=1}].
\end{align*}

Since $\delta>0$ is arbitrary, we obtain the representation \eqref{eua} for any $t>0$.
\vskip0.5cm
{\bf Acknowledgements.} I would like to express my sincere gratitude to Prof. Errico Presutti for his great ideas which help me a lot to complete this paper.
\bibliographystyle{amsalpha}

\vskip 0.3cm
Thu Dang Thien Nguyen\\
             Gran Sasso Science Institute, Viale Francesco Crispi, 7, L'Aquila  67100, Italy\\
			 Department of Mathematics, University of Quynhon, Quy Nhon, Vietnam\\
             Email: thu.nguyen@gssi.it         

\end{document}